\pdfoutput=1

\documentclass[10pt]{article}
\usepackage{amsmath}
\usepackage{amsfonts}
\usepackage{amssymb}
\usepackage{graphicx}
\usepackage{mathdots}
\usepackage{lscape}
\newtheorem{theorem}{Theorem}

\newtheorem{corollary}[theorem]{Corollary}
\newtheorem{lemma}[theorem]{Lemma}
\newtheorem{definition}[theorem]{Definition}
\newenvironment{proof}[1][Proof]{\noindent\textbf{#1.} }{\ \rule{0.5em}{0.5em}}
\newenvironment{t_enumerate}{
\begin{enumerate}
\setlength{\itemsep}{1pt}
\setlength{\parskip}{0pt}
\setlength{\parsep}{0pt}}{\end{enumerate}
}

\numberwithin{equation}{section}
\numberwithin{theorem}{section}
\begin{document}

\date{}
\title{Isotropic foliations of coadjoint orbits\\from the Iwasawa decomposition}
\author{William D. Kirwin\thanks{Mathematisches Institut, University of Cologne, Weyertal 86--90, Cologne, Germany.\newline Email: will.kirwin@gmail.com}}
\maketitle

\begin{abstract}
Let $G$ be a noncompact real semisimple Lie group. The regular coadjoint orbits of $G$ can be partitioned into a finite set of types. We show that on each regular orbit, the Iwasawa decomposition induces a left-invariant foliation which is isotropic with respect to the Kirillov symplectic form. Moreover, the leaves are affine subspaces of the dual of the Lie algebra, and the dimension of the leaves depends only on the type of the orbit. When $G$ is a split real form, the foliations induced from the Iwasawa decomposition are actually Lagrangian fibrations with a global transverse Lagrangian section.
\end{abstract}

\medskip\noindent\textbf{Keywords}:
\begin{tabular}
[t]{l}%
coadjoint orbit, Iwasawa decomposition, isotropic foliation, Lagrangian fibration,\\
Cartan subalgebra, conjugacy class 
\end{tabular}


\section{Introduction}

Let $G$ be a noncompact real semisimple Lie group with Lie algebra $\mathfrak{g.}$ The Iwasawa decomposition of $G$ is a group decomposition $G=NAK$, where $N$ is a nilpotent analytic\footnote{Recall that in the context of Lie groups, an analytic subgroup is simply a connected Lie subgroup.} subgroup, $A$ is an abelian analytic subgroup, and $K$ is a maximal compact analytic subgroup \cite{Iwasawa49}; it is unique up to conjugation.

Let $X\in\mathfrak{g}^{\ast}$ be a regular point in the dual of the Lie algebra, and denote by $\Omega=Ad_{G}^{\ast}X$ the coadjoint orbit through $X$. Since $G$ acts transitively on $\Omega,$ the coadjoint orbit is a homogeneous space $\Omega\simeq G/G_{X}.$ For a regular point $X$, the stabilizer subgroup is always conjugate to a subgroup of $AK$, whence the nilpotent part $N$ acts on the coadjoint orbit $\Omega$, yielding a foliation of $\Omega$ by $N$-orbits.

Each coadjoint orbit $\Omega$ carries a canonical symplectic structure $\omega$ called the Kirillov symplectic structure. In this paper, we investigate the interplay between the Iwasawa decomposition, in particular the foliations of coadjoint orbits by $N$-orbits, and the Kirillov symplectic structure. In general, the $N$-orbits do not have any nice symplectic structure, i.e., they are not necessarily symplectic, isotropic, coisotropic, or Lagrangian, nor are they well behaved with respect to the transitive left $G$ action on $\Omega$ (since $N$ and $AK$ do not commute in general).

On the other hand, we will show here that there is always a nilpotent analytic subgroup of $N$ (in some cases $N$ itself) whose orbits in $\Omega$ yield a left-invariant isotropic foliation of $\Omega$. The set of regular orbits can be partitioned into a finite set of classes, corresponding to the conjugacy class of the stabilizer, which we refer to as the \emph{type} of $\Omega.$ The subgroup of $N$ which induces a left-invariant isotropic foliation of $\Omega$ is completely determined by the type of $\Omega$.

Moreover, when $\Omega$ is considered as an embedded submanifold of $\mathfrak{g}^{*}$ in the natural way, the leaves of this left-invariant isotropic foliation are affine subspaces of $\mathfrak{g}^{*}$. Classically, foliations of surfaces by affine lines are known as \emph{rulings} (similarly or foliations of complex surfaces by complex lines), so we call these
foliations $n$-rulings, where $n$ is the dimension of the leaves. Indeed, for $G=SL(2,\mathbb{R})$, the foliations induced by the (two) choices of Iwasawa decomposition yield exactly the two rulings of the one-sheeted hyperboloid by lines. It sometimes happens, as in the case of the one-sheeted hyperboloid-type orbits of $SL(2,\mathbb{R})$, that the left-invariant isotropic foliation induced by a subgroup of $N$ is actually a Lagrangian fibration. For such a coadjoint orbit $\Omega=Ad_{G}^{*}\xi$, the submanifold $Ad_{K}^{*}\xi$ turns out to be a global  transverse Lagrangian.

\bigskip
The original motivation for this work was to find examples of Lagrangian fibrations with noncompact fibers, for example to apply the techniques of Lagrangian Floer homology (see, for example, \cite{Abbondandolo-Schwarz},\cite{Salamon-Weber},\cite{Viterbo}) for cotangent bundles in new settings. As it turns out, any of the left-invariant isotropic foliations constructed here which are Lagrangian also admit a transverse Lagrangian submanifold and hence are symplectomorphic to the cotangent bundle of that submanifold, and so in some sense noncompact coadjoint orbits do not furnish ``new'' examples of Lagrangian fibrations with noncompact fibers. On the other hand, for each of these cotangent-type coadjoint orbits, the nonuniqueness of the Iwasawa decomposition yields several distinct symplectomorphisms with the cotangent bundle of the transverse submanifold, and the Weyl group acts transitively on the set of the cotangent-type Lagrangian fibrations. Presumably, this means that the Floer homologies are representations of the associated Weyl group, which may lead to interesting additional structure in the the homology groups, but we will not address these points in this note.

Of course, coadjoint orbits, and in particular their symplectic geometry, have been the subject of much scrutiny in the literature, and we cannot hope here to give a complete list of references. A good general reference for the geometry of coadjoint orbits is the book \cite{Duistermaat-Kolk} by Duistermaat and Kolk. In \cite{Bernatska-Holod}, Bernatska and Holod study the coadjoint orbits of compact Lie groups (a focus which is somewhat complementary to the one taken here).

Our classification of regular coadjoint orbits is based on an invariant of conjugacy classes of Cartan subalgebras which was introduced by Sugiura in \cite{Sugiura} and \cite{Sugiura-correction}. In \cite{KostantCSA}, Kostant also studied the problem of classification of conjugacy classes of Cartan subalgebras and provided a set of theorems which achieves essentially the same result as Sugiura, although the invariants he constructs are slightly more complicated, and Sugiura's invariant is quite well suited to our application. In \cite{Rothschild}, Rothschild classifies the orbits in a real reductive Lie algebra in terms of the intersection of the Lie algebra with the orbits in its complexification.

\bigskip
We give here an outline of the rest of the paper. In the remainder of the introduction, we will state our main results more precisely.

In Section \ref{sec:setup}, we recall (mostly to fix notation) some standard material regarding real Lie algebras. First, we specify the algebras which we will study, then we recall Cartan subalgebras and their basic properties. This allows us to define regular coadjoint orbits. Section \ref{sec:CSAs+roots} recalls the Iwasawa decomposition. Finally, Section \ref{sec:cayley} reviews the theory of Cayley transforms and Sugiura's invariant, which allows us to classify coadjoint orbits according to \emph{type}; since much of the material here is less standard (and indeed since we know of no reference for some of the results), much more detail is included. The section culminates with Theorem \ref{lemma:hj-decomp}, which describes the structure of an arbitrary $\theta$-stable Cartan subalgebras in terms of the root data associated to a maximally noncompact Cartan subalgebra.

Section \ref{sec:foliations} is the heart of the paper. In Section \ref{subsec:maxNC}, we construct the isotropic foliations of ``maximally noncompact'' coadjoint orbits induced from the Iwasawa decomposition. We also classify those orbits for which the induced foliation is actually a Lagrangian fibration. Then, we consider the general case in Section \ref{subsec:nonmaxNC}. Finally, in Section \ref{subsec:type-1}, we show that if the dimension of $G$ is not too small, there are at least two distinct types of orbits which admit nontrivial isotropic foliations induced by the Iwasawa decomposition.

In Section \ref{sec:rulings}, we show that the foliations are actually rulings; that is, that the leaves are affine subspaces of $\mathfrak{g}^{\ast}.$

The paper concludes with an Appendix which contains the specific geometric data of the isotropic foliations for each of the simple real Lie algebras.

\subsection{Main results.}

In the remainder of the introduction, we assume the reader is familiar with the structure theory of simple Lie algebras; the notation and relevant material for the rest of the paper will be introduced in Section \ref{sec:setup}.

Let $G$ be a connected Lie group with real, semisimple Lie algebra $\mathfrak{g}$. By Cartan's criterium, a Lie algebra $\mathfrak{g}$ is semisimple if and only if the Killing form is nondegenerate. We henceforth identify $\mathfrak{g}$ with its dual by the Killing form. Hence, we consider adjoint orbits with symplectic structure equal to the pullback of the Kirillov symplectic form on the coadjoint orbits via the identification.

Let $X\in\mathfrak{g}$ be a regular point, and denote by $\Omega$ the orbit through $X$. Such an orbit is called a regular orbit, and in this article we restrict our attention to such orbits. The stabilizer $\mathfrak{h}$ of $\Omega$ at $X$ is a Cartan subalgebra, and without loss of generality (by simply choosing another point in the orbit) we may assume it to be stable under some choice $\theta$ of Cartan involution. Let $\mathfrak{h}_{0}$ be a maximally noncompact $\theta$-stable Cartan subalgebra, and let $\Delta$ denote the set of roots with respect to $\mathfrak{h}_{0}^{\mathbb{C}}.$

In \cite{Sugiura}, Sugiura shows that conjugacy classes of Cartan subalgebras are in one-to-one correspondence with conjugacy classes of real admissible roots systems. We denote by $\mathbf{F}=\{\alpha_{1},\dots,\alpha_{j}\} $ such a root system. There is a unique conjugacy class of Cartan subalgebras associated to $\mathbf{F}=\emptyset,$ the so-called maximally noncompact Cartan subalgebras. There is also a unique conjugacy class of Cartan subalgebras associated to (the unique conjugacy class of) the maximal real admissible root system. As there is a finite number of conjugacy classes of real admissible root systems, we may partition the set of regular orbits into a finite set of types.

\begin{definition}
The orbit $\Omega$ is said to be of \emph{\textbf{type}} $\mathbf{F}$ if the admissible root system associated to the conjugacy class of stabilizers of $\Omega$ is $\mathbf{F}$. If $\mathbf{F}=\emptyset$, then $\Omega$ is said to be of \emph{\textbf{type}} $\mathbf{0}$.
\end{definition}

Choose an Iwasawa decomposition $G=KAN$ (compatible with the Cartan involution with respect to which $\mathfrak{h}$ is stable), or, at the Lie algebra level, $\mathfrak{g}=\mathfrak{k}\oplus\mathfrak{a}\oplus\mathfrak{n} $. Our first main result is the following theorem.

\setcounter{section}3\setcounter{theorem}{0}
\begin{theorem}
Each type-$0$ orbit $\Omega$ admits a left-invariant isotropic foliation $\mathcal{N}$ with $\mathcal{N}_{X}=\mathfrak{n}$.
\end{theorem}
\setcounter{section}1\setcounter{theorem}{0}

We can go further and completely characterize those orbits for which the induced foliation is actually a Lagrangian fibration.

\setcounter{section}3\setcounter{theorem}{1}
\begin{theorem}
If $\mathfrak{g}$ is a split real form or a complex simple Lie algebra regarded as a real Lie algebra, then the left-invariant isotropic foliation of Theorem \ref{thm:foliation} is a left $G$-invariant Lagrangian $\mathbb{R}^{n}$-fibration over $Ad_{K}X,$ where $n=\dim\mathfrak{n}$, and $Ad_{K}X$ is a transverse Lagrangian submanifold.
\end{theorem}
\setcounter{section}{1}\setcounter{theorem}{0}

For type-$\mathbf{F}$ orbits with $|\mathbf{F}|>0$, we can show the existence of left-invariant isotropic foliations whose tangent space at $X$ is a subspace of $\mathfrak{n}$. Indeed, using the theory of Cayley transforms, we give in Theorem \ref{thm:general case} a procedure by which one can use the foliation of Theorem \ref{thm:foliation} above to induce left-invariant isotropic foliations on type-$\mathbf{F}$ orbits for any $\mathbf{F}$; because of the technical nature of the result, we leave the precise statement to Section \ref{subsec:nonmaxNC}.

The subspace of $\mathfrak{n}$ which is tangent to the isotropic left-invariant foliation might be (and is, in some cases) $\{0\}$, which implies the associated foliation is trivial. On the other hand, if the dimension of $G$ is not too small, we have the following result, which for the sake of readability in the introduction is a simplification of Theorem \ref{thm:h1iso}. See the remarks following Theorem \ref{thm:h1iso} for the precise meaning of ``too small.''

\setcounter{section}3\setcounter{theorem}{6}
\begin{theorem}
If $\Omega$ is type $\{\alpha\}$, with $\alpha$ simple, and the dimension of $G$ is not too small, then $\Omega$ admits a left-invariant isotropic foliation $\mathcal{N}$ with $\{0\}\neq\mathcal{N}_{X}\varsubsetneq\mathfrak{n}$.
\end{theorem}
\setcounter{section}1\setcounter{theorem}{0}

Our final result ist that the induced foliations are (generalized) rulings.

\setcounter{section}{4}\setcounter{theorem}{0}
\begin{theorem}
Let $\Omega$ be a coadjoint orbit in $\mathfrak{g}$ through a regular point $X$. The leaves of the left-invariant isotropic foliation of $\Omega$ of Theorem \ref{thm:general case} are affine subspaces of $\mathfrak{g}$.
\end{theorem}
\setcounter{section}1\setcounter{theorem}{0}

\section{Setup\label{sec:setup}}

Much of the material in this section is standard; we recall it to set notation
and fix the range of our focus. The more standard material can be found in any
of the excellent texts \cite{Knapp}, \cite{Helgason}, \cite{Varadarajan}. For
less standard material, we include specific references or short proofs where
no reference could be found.

\bigskip

The foliations we want to describe arise from the Iwasawa decomposition of a semisimple Lie algebra (the group level decomposition will also play a role). Since we are interested in the symplectic geometry of coadjoint orbits, we will restrict our attention to coadjoint orbits in real semisimple Lie algebras, though in some cases the underlying algebra has a complex structure. Every semisimple Lie algebra is a direct sum of simple Lie algebras, so we further restrict our attention to simple real Lie algebras. These come in two types: simple complex Lie algebras regarded as real algebras, and noncomplex simple real Lie algebras. If $\mathfrak{g}$ is a noncomplex simple real Lie algebra, then $\mathfrak{g}^{\mathbb{C}}$ is a simple complex Lie algebra.

Our construction arises from the Iwasawa decomposition, which in the compact case is trivial in such a way that our construction yields only the trivial foliation by points.\footnote{The Iwasawa decomposition of a compact Lie group $G$ is $G=GAN$ where $A=N=\{e\}.$ The foliations we construct are induced from $N$, and hence are trivial for compact groups.} Hence, we will consider only
\begin{t_enumerate}
\item simple complex Lie algebras, regarded as real Lie algebras (which are never compact), and
\item noncompact noncomplex simple real Lie algebras.
\end{t_enumerate}

The simple complex Lie algebras are the four classical series of types $A_{n},~B_{n},~C_{n},$~$D_{n},$ and the exceptional algebras of types $E_{6},~E_{7},~E_{8},~F_{2},~F_{4},$ and $G_{2}$. The list of noncompact noncomplex simple real Lie algebras is longer; see Figure \ref{fig:LAdata} in the Appendix.

Let $G$ be a simple Lie group with Lie algebra $\mathfrak{g}$. Throughout, we will denote by $\mathfrak{g}$ a simple real Lie algebra, $\mathfrak{g}^{\mathbb{C}}$ its complexification, and $\mathfrak{s}$ a simple complex Lie algebra. The Killing form $B$ on $\mathfrak{g}$ is nondegenerate and hence induces an isomorphism $\mathfrak{g}^{\ast}\simeq\mathfrak{g}$, by which we identify $\mathfrak{g}^{\ast}$ with $\mathfrak{g.}$

Fix $\xi\in\mathfrak{g}^{\ast}$ and consider the coadjoint orbit $W:=Ad_{G}^{\ast}\xi$. Denote the stabilizer group of $\xi$ by $G_{\xi}:=\{g\in G:Ad_{g}^{\ast}\xi=\xi\},$ which has Lie algebra
\[
\mathfrak{g}_{\xi}=\{Z\in\mathfrak{g}:ad_{Z}^{\ast}\xi=0\}.
\]

The tangent space to $W$ at the point $\xi$ is the set $\{ad_{X}^{\ast}\xi:X\in\mathfrak{g}\}$, and the Kirillov symplectic form on $W$ is determined by left invariance and it's value at $\xi$, which is
\[
\omega_{\xi}(ad_{X}^{\ast}\xi,ad_{Y}^{\ast}\xi):=\xi([X,Y]).
\]

By Cartan's criterion, the Killing form $B(X,Y)=tr(ad_{X}\,ad_{Y})$ is nondegenerate if $\mathfrak{g}$ is semisimple, and so we can use the Killing form $B$ to transfer the coadjoint orbit $W$ and its symplectic structure to $\mathfrak{g.}$ Define
\[
\Omega:=B^{-1}(W)=Ad_{G}X.
\]
Under this identification, the stabilizer algebra of $X$ is $B$-orthogonal to the tangent space of $\Omega$ at $X$.

\bigskip
We now turn to regular orbits. For $X\in\mathfrak{g}$, consider the characteristic polynomial 
\begin{equation}
\det(\lambda\mathbf{1}-ad_{X})=\lambda^{n}+\sum_{j=0}^{n-1}p_{j}(X)\lambda^{j}. \label{eqn:charpoly}
\end{equation}
Recall that $X$ is said to be \emph{regular} if $p_{r}(X)\neq0$, where $r$ is the rank of $G$. Regular points are thus generic in $\mathfrak{g}.$ An orbit $\Omega=Ad_{G}X$ is \emph{regular} if $X$ is regular.

The polynomials $p_{j}$ are invariant under any automorphism of $\mathfrak{g}$ \cite[Sec. 3.9]{Varadarajan}, in particular, $p_{r}(Ad_{g}X)=p_{r}(X)$. This means $Ad_{g}X$ is regular if and only if $X$ is regular. Moreover, if $\det_{\mathfrak{g}^{\mathbb{C}}}(\lambda\mathbf{1}-ad_{\mathfrak{g}^{\mathbb{C}}}X)=\lambda^{n}+\sum_{j=0}^{n-1}p_{j}^{\prime}(X)\lambda^{j}$, then $p_{j}=\left.  p_{j}^{\prime}\right\vert _{\mathfrak{g}}$, which shows that $X\in\mathfrak{g}$ is regular if and only if $X=X+0i\in\mathfrak{g}^{\mathbb{C}}$ is regular.

For $X$ regular, the generalized $0$-eigenspace
\begin{equation}
\mathfrak{g}_{X}:=\{Y\in\mathfrak{g}:ad_{X}^{s}Y=0\text{ for some }s\geq1\}\label{eqn:gend0espace}%
\end{equation}
is called a Cartan subalgebra, which we abbreviate by \emph{CSA}. If $\mathfrak{g}$ is a semisimple Lie algebra, then one may take $s=1$, that is, $\mathfrak{g}_{X}=\{Y\in\mathfrak{g}:[X,Y]=0\}$ corresponds to the stabilizer algebra of $X$. Since we will only be concerned with stabilizers of regular points in semisimple Lie algebras, we have denoted both by the same symbol $\mathfrak{g}_{X}.$

If $\mathfrak{g}$ is a complex simple Lie algebra, then $\mathfrak{g}_{X}$ is a CSA if and only if it is a maximal abelian subalgebra such that $\{ad_{Y}:Y\in\mathfrak{g}_{X}\}$ is simultaneously diagonalizable. If $\mathfrak{g}$ is a real semisimple Lie algebra, then $\mathfrak{g}_{X}$ is a CSA if and only if the complexification $\mathfrak{g}_{X}^{\mathbb{C}}$ is a CSA in $\mathfrak{g}^{\mathbb{C}}.$ All CSAs have the same dimension, which is the rank $r$ of $G$.

From the fact that $\mathfrak{g}_{X}$ is a CSA for $X$ regular and the $Ad$-invariance of regularity discussed above, we make the following observation which is basic to our construction.

\begin{lemma}
\label{lemma:stabCSA}If $\Omega$ is a regular orbit, then for each $Y\in\Omega$, $\mathfrak{g}_{Y}$ is a Cartan subalgebra.
\end{lemma}

\subsection{The Iwasawa decomposition.\label{sec:CSAs+roots}}

Let $\mathfrak{s}$ be a complex simple Lie algebra. For a CSA $\mathfrak{s}_{X}<\mathfrak{s}$, denote the associated set of roots by $\Delta=\Delta(\mathfrak{s},\mathfrak{s}_{X})$, a choice of positive roots by $\Delta^{+}$, and a choice of simple roots by $\Pi$. For each nonzero root $\alpha\in\Delta$, denote the $B$-dual vector by $H_{\alpha}\in\mathfrak{s}_{X}$.

For $H\in\mathfrak{s}_{X}$, the polynomial $p_{r}$ appearing in (\ref{eqn:charpoly}) can be written in terms of the roots as $p_{r}(H)=\Pi_{\alpha\in\Delta}\alpha(H)$, so that a vector $H\in\mathfrak{h}$ is regular if and only if $\alpha(H)\neq0$ for each root $\alpha$. In particular, the \textit{singular} vectors in $\mathfrak{h}$ (i.e. vectors which are not regular) form a codimension-$1$ subvariety consisting of the union of the kernels of the roots.

\bigskip
For now, we assume $\mathfrak{g}$ is a real semisimple Lie algebra (not necessarily simple), and fix a Cartan involution $\theta$. Let $\mathfrak{g}=\mathfrak{k}\oplus\mathfrak{p}$ be the corresponding Cartan decomposition into $\pm1$ $\theta$-eigenspaces, where the $+1$-eigenspace $\mathfrak{k}$ is a maximal compact subalgebra of $\mathfrak{g}$. Denote by $K$ the corresponding analytic subgroup of $G$ with Lie algebra $\mathfrak{k}$.

If $\mathfrak{h}$ is a $\theta$-stable CSA, then
\[
\mathfrak{h}=(\mathfrak{h}\cap\mathfrak{k})\oplus(\mathfrak{h}\cap\mathfrak{p}).
\]
For a $\theta$-stable CSA $\mathfrak{h}$, the subspace $\mathfrak{h}^{+}:=\mathfrak{h}\cap\mathfrak{k}=\{X\in\mathfrak{h}:ad_{X}$ has purely imaginary eigenvalues$\}$ is called the \emph{toroidal part}, and $\mathfrak{h}^{-}:=\mathfrak{h}\cap\mathfrak{p}=\{X\in\mathfrak{h}:ad_{X}$ has real eigenvalues$\}$ is called the \emph{vector part.}

Every CSA is conjugate to a $\theta$-stable CSA. For a $\theta$-stable CSA, the dimension of $\mathfrak{h}^{+}$ is called the \textit{compact dimension}, which we will denote by $\mathfrak{k}$-$\dim\mathfrak{h}$, and the dimension of $\mathfrak{h}^{-}$ is called the noncompact dimension, which we will denote by $\mathfrak{p}$-$\dim\mathfrak{h}$. All maximally compact CSAs are conjugate, as are all maximally noncompact CSAs. For nonmaximal CSAs, the conjugacy classes can be enumerated by conjugacy classes of admissible root systems (see Theorem \ref{thm:Sugiura} below).

\bigskip
Fix a maximal abelian subalgebra $\mathfrak{a}\subset\mathfrak{p}$, and denote the associated system of restricted roots by $\Sigma$. The dimension of $\mathfrak{a}$ is called the \textit{real rank}, denoted $\mathbb{R}$-$\operatorname*{rank}\mathfrak{g}$, and is independent of the choice of $\mathfrak{a}$. Choose a set of positive restricted roots $\Sigma^{+},$ and also a set of simple restricted roots $\Sigma_{0}^{+}$.

Let the centralizer of $\mathfrak{a}$ in $\mathfrak{k}$ be $\mathfrak{m}=Z_{\mathfrak{k}}(\mathfrak{a}),$ and choose a maximal abelian subspace $\mathfrak{t}\subset\mathfrak{m}$. Then
\[
\mathfrak{h}_{0}:=\mathfrak{t}\oplus\mathfrak{a}
\]
is a maximally noncompact CSA of $\mathfrak{g}$. The roots of $\mathfrak{g}^{\mathbb{C}}$ with respect to $\mathfrak{h}_{0}^{\mathbb{C}},$ when restricted to $\mathfrak{a}$, yield the restricted roots. Moreover, one may choose the sets of positive and simple roots so that positive roots restrict to positive restricted roots, simple roots restrict to simple restricted roots, and the restriction map $\alpha\in\Delta,\Delta^{+},\Pi\mapsto\left.\alpha\right\vert _{\mathfrak{a}}\in\Sigma,\Sigma^{+},\Sigma_{0}^{+}$ (resp.) is onto. We will always make such compatible choices.

\bigskip
\noindent\textbf{Remark.} If $\mathfrak{g}$ is a complex simple Lie algebra, regarded as a real Lie algebra, then the restricted root system of $\mathfrak{g}$ as a real Lie algebra is equal to the roots system of $\mathfrak{g}$ considered as a complex Lie algebra, that is, $\Sigma(\mathfrak{g},\mathfrak{a})=\Delta(\mathfrak{g}^{\mathbb{C}},\mathfrak{a}^{\mathbb{C}})$, and the restricted root spaces have real dimension $2$ \cite[p. 153]{Klimyk-Vilenkin}.

\bigskip
Let $\mathfrak{n}:=\bigoplus_{\lambda\in\Sigma^{+}}\mathfrak{g}_{\lambda}$. The \emph{Iwasawa decomposition} of $\mathfrak{g}$ is the direct sum decomposition
\[
\mathfrak{g=k}\oplus\mathfrak{n}\oplus\mathfrak{a.}
\]
The subspace $\mathfrak{n}$ is a nilpotent subalgebra, and $\mathfrak{n}\oplus\mathfrak{a}$ is a solvable subalgebra such that $[\mathfrak{n}\oplus\mathfrak{a},\mathfrak{n}\oplus\mathfrak{a}]\subset\mathfrak{n}$.

The Iwasawa decomposition of $\mathfrak{g}$ induces a global decomposition $G=KNA$ (that is, group multiplication yields a diffeomorphism $K\times N\times A$ onto $G$), where $N$ and $A$ are simply connected analytic subgroups of $G$ with Lie algebras $\mathfrak{n}$ and $\mathfrak{a}$ (resp.). There is some, but not total, freedom in the order of the analytic subgroups; it is necessary that $N$ and $A$ be taken together. So for instance, $G$ equals $NAK$ or $KAN$, but not in general $AKN$.

\subsection{Structure theory of Cartan subalgebras and Cayley transforms.
\label{sec:cayley}}

In this section, $\mathfrak{g}$ denotes a noncomplex real simple Lie algebra.\footnote{For a complex simple Lie algebra, regarded as a real Lie algebra, there exists a unique conjugacy class of Cartan subalgebras, thus obviating the need for Cayley transforms.} In \cite{Sugiura}, Sugiura describes an invariant of conjugacy classes of Cartan subalgebras. We recall this invariant below. Cayley transforms provide a way to enumerate and move between different conjugacy classes of CSAs, essentially by exploiting the fact that all CSAs are conjugate in the complexification $\mathfrak{g}^{\mathbb{C}}$. Using Cayley transforms, we are able to use Sugiura's invariant to prove Theorem \ref{lemma:hj-decomp}, which elucidates the structure of each conjugacy class of Cartan subalgebras.

\bigskip
Recall that a root $\alpha\in\Delta(\mathfrak{g}^{\mathbb{C}},\mathfrak{h}^{\mathbb{C}})\subset(\mathfrak{h}^{\mathbb{C}})^{\ast}$ is \emph{real} if it takes real values on $\mathfrak{h}$. Since roots take real values on $\mathfrak{h}^{-}$ and $i\mathfrak{h}^{+}$, a root is real if and only if it restricts to $0$ on $\mathfrak{h}^{+}$. In terms of dual vectors, a root $\alpha$ is real if and only if $H_{\alpha}\in\mathfrak{h}^{-}$. Let $H\in\mathfrak{h}^{\mathbb{C}}$. If $\alpha$ is real, then $\alpha(\theta(H_{\left(  \mathfrak{h}^{+}\right)^{\mathbb{C}}}+H_{\left(\mathfrak{h}^{-}\right)  ^{\mathbb{C}}}))=-\alpha(H_{\left(  \mathfrak{h}^{-}\right)  ^{\mathbb{C}}})$ so that $\theta\alpha=-\alpha$.

Two roots $\alpha,\beta$ are said to be \emph{strongly orthogonal} if they are orthogonal and $\alpha\pm\beta$ is not a root.

\bigskip
\noindent\textbf{Remark.} In fact, if $\alpha\neq\pm\beta$, then $\alpha$ is orthogonal to $\beta$ if and only if neither $\alpha+\beta$ nor $\alpha-\beta$ is a root.\cite[Prop. 2.48(f)]{Knapp}

\bigskip
As we will see, our construction requires a sets of strongly orthogonal real roots. Maximal sets of strongly orthogonal roots in complex simple Lie algebras have been classified by Agaoka and Kaneda in \cite{Agaoka-Kaneda}. Given such a set, one could then find subsets consisting of real roots. Another possible approach would be to observe that sets of strongly orthogonal real roots are in some sense dual to the sets of noncompact imaginary roots,\footnote{Using Lemma \ref{lemma:Cayley_props} and similar techniques, applying the Cayley transforms associated to a a maximal sequence of noncompact imaginary roots results in a maximal sequence of strongly orthogonal real roots.} which play a role in the classification of real simple Lie algebras (see \cite[Sec VI.11]{Knapp}, for example). Algorithms are known which generate maximal sequences of noncompact imaginary roots, and we could then use Cayley transforms to convert such a sequence into a a maximal sequence of strongly orthogonal real roots.

We will take a different approach, based on Sugiura's classification of conjugacy classes of CSAs, which we introduce now \cite[Def. 9]{Sugiura}. A set $\mathbf{F}=\{\alpha_{1},\dots,\alpha_{\ell}\}$ of distinct positive roots is called an \emph{admissible root system} if $\alpha_{i}\pm\alpha_{j}\not \in \Delta,~1\leq i,j\leq\ell$.

Two admissible root systems $\mathbf{F}_{1}$ and $\mathbf{F}_{2}$ are said to be equivalent if the dual vectors span the same subspace, that is, if $\bigoplus_{\alpha\in\mathbf{F}_{1}}\mathbb{R}H_{\alpha}\mathbb{=}\bigoplus_{\alpha\in\mathbf{F}_{2}}\mathbb{R}H_{\alpha}$. Two admissible root systems are said to be conjugate if there exists an element $s\in W(\mathfrak{g}^{\mathbb{C}},\mathfrak{h}^{\mathbb{C}})$ of the Weyl group such that $s\mathbf{F}_{1}$ is equivalent to $\mathbf{F}_{2}.$ In \cite{Sugiura}, Sugiura proved the following classification theorem.

\begin{theorem}
\label{thm:Sugiura}\cite[Thm. 6]{Sugiura} There is a one-to-one correspondence between conjugacy classes of CSAs in a real semisimple Lie algebra $\mathfrak{g}$ and conjugacy classes of admissible roots systems which are contained in $\mathfrak{a}$.
\end{theorem}

Moreover, Sugiura showed that if $\mathbf{F}$ is the admissible root system associated to a conjugacy class of CSAs, then there is a $\theta$-stable CSA $\mathfrak{h}_{\mathbf{F}}$ in the conjugacy class such that
\[
(\mathfrak{h}_{\mathbf{F}}\cap\mathfrak{a)}^{\perp}\cap\mathfrak{a}=\bigoplus_{\alpha\in\mathbf{F}}\mathbb{R}H_{\alpha}.
\]
That is, the conjugacy class of $\mathfrak{h}_{\mathbf{F}}$ is determined by the (roots which span the) complement of the vector part of $\mathfrak{h}_{\mathbf{F}}$ in $\mathfrak{a}$. Since the roots in an admissible root system associated to a CSA are necessarily real, we refer to such a system as a real admissible root system.

\bigskip
We recall now the Cayley transform that we will need for our construction, which takes a $\theta$-stable CSA $\mathfrak{h}$ and a real root $\alpha$ and produces a $\theta$-stable CSA $D_{\alpha}(\mathfrak{h})$ with $\mathfrak{k}$-$\dim D_{\alpha}(\mathfrak{h})=\mathfrak{k}$-$\dim\mathfrak{h}+1$. We loosely follow the description given in \cite[Sec. VI.7]{Knapp}.

If $\alpha$ is a real root, then $\bar{\alpha}=\alpha$ and so we can choose $0\neq X_{\alpha}\in\mathfrak{g}_{\alpha}\cap\mathfrak{g}$ and normalize it to obtain $E_{\alpha}$ such that $B(E_{\alpha},\theta E_{\alpha})=-2/\left\vert\alpha\right\vert ^{2}.$ Define
\[
\mathbf{d}_{\alpha}:=Ad\left(  \exp\left(  \frac{i\pi}{4}(\theta E_{\alpha}-E_{\alpha})\right)  \right)  .
\]
The $D_{\alpha}$-transform of $\mathfrak{h}$ is $D_{\alpha}\mathfrak{h}:=\mathbf{d}_{\alpha}(\mathfrak{h}^{\mathbb{C}})\cap\mathfrak{g}$, and one may compute that $\mathbf{d}_{\alpha}\left(  \frac{2}{\left\vert \alpha\right\vert^{2}}H_{\alpha}\right)  =i(E_{\alpha}+\theta E_{\alpha})$ from which it follows that
\begin{equation}
D_{\alpha}\mathfrak{h}=\ker\left(  \left.  \alpha\right\vert _{\mathfrak{h}}\right)  \oplus\mathbb{R}(E_{\alpha}+\theta E_{\alpha}).\label{eqn:Cayley_CSA}
\end{equation}
$D_{\alpha}\mathfrak{h}$ is clearly $\theta$-stable, and since $E_{\alpha}+\theta E_{\alpha}$ is in $\mathfrak{k}$, we see that $\mathfrak{p}$-$\dim D_{\alpha}\mathfrak{h}=\mathfrak{p}$-$\dim\mathfrak{h}-1$. For each root $\gamma$ with respect to $\mathfrak{h}$, we obtain roots $\mathbf{d}_{\alpha }(\gamma)=\gamma\circ\mathbf{d}_{\alpha}^{-1}$ with respect to $D_{\alpha }\mathfrak{h}$.

A CSA $\mathfrak{h}$ is maximally noncompact if and only if there exist no noncompact imaginary roots, and maximally compact if and only if there exist no real roots.

\bigskip
We can use the real admissible root system associated to a CSA $\mathfrak{h}$ to generate a series of Cayley transforms which transform a $\mathfrak{h}_{0} $ into $\mathfrak{h}$. This process reveals the inner structure of the CSA $\mathfrak{h}$. We need a few more facts regarding real and strongly orthogonal roots, and in particular their behavior under Cayley transforms.

\begin{lemma}
\label{lemma:Cayley_props}Let $\alpha,\beta,\gamma\in\Delta(\mathfrak{g}^{\mathbb{C}},\mathfrak{h}^{\mathbb{C}})$.
\begin{t_enumerate}
\item If $\alpha$ is real and orthogonal to $\beta,$ then $H_{\mathbf{d}_{\alpha}\beta}=H_{\beta}$ and $\left.  \mathbf{d}_{\alpha}\beta\right\vert_{\ker\alpha}=\beta.$
\item If $\alpha,\beta,\gamma$ are orthogonal and $\alpha$ is real, then $\mathbf{d}_{\alpha}\beta$ is orthogonal to $\mathbf{d}_{\alpha}\gamma$. Hence, if $\beta$ is real then so is $\mathbf{d}_{\alpha}\beta$.
\item If $\alpha$ is real and $\beta$ is strongly orthogonal to $\gamma$, then $\mathbf{d}_{\alpha}\beta$ is strongly orthogonal to $\mathbf{d}_{\alpha}\gamma$.
\item If $\alpha$ is real and strongly orthogonal to $\beta$, then $E_{\mathbf{d}_{\alpha}\beta}=E_{\beta}$.
\end{t_enumerate}
\end{lemma}

\begin{proof}
(1) If $H\in\ker\left.  \alpha\right\vert _{\mathfrak{h}},$ then $\mathbf{d}_{\alpha}(H)=H.$ Hence $B(H_{\beta},H)=\beta(H)=\beta (\mathbf{d}_{\alpha}^{-1}H)=\mathbf{d}_{\alpha}\beta(H)$, which proves the second part of the claim. We then compute $B(H_{\beta},E_{\alpha}+\theta E_{\alpha})=0$ since $\mathfrak{h}^{\mathbb{C}}$ is $B$-orthogonal to $\mathfrak{g}_{\alpha} $ and $\mathfrak{g}_{-\alpha}.$ On the other hand, $\mathbf{d}_{\alpha}\beta(E_{\alpha}+\theta E_{\alpha})=B(H_{\beta},\mathbf{d}_{\alpha}^{-1}(E_{\alpha}+\theta E_{\alpha}))=B(H_{\beta},\frac {2}{i\left\vert \alpha\right\vert ^{2}}H_{\alpha})=0$. Hence, $B(H_{\beta},H)=\mathbf{d}_{\alpha}\beta(H)$ for each $H\in D_{\alpha}\mathfrak{h}$, which yields the first claim.

(2) Since $\beta$ and $\gamma$ are orthogonal, we have $B(\mathbf{d}_{\alpha}\beta,\mathbf{d}_{\alpha}\gamma)=B(H_{\beta},H_{\gamma})=0$, whence $\mathbf{d}_{\alpha}\beta$ is orthogonal to $\mathbf{d}_{\alpha}\gamma$. As noted in the proof of part (1) above, $\mathbf{d}_{\alpha}\beta(E_{\alpha }+\theta E_{\alpha})=0$. Hence, if $\beta$ is real then $\mathbf{d}_{\alpha }\beta$ vanishes on each summand of $D_{\alpha}\mathfrak{h}\cap\mathfrak{k}=\mathbb{R}(E_{\alpha}+\theta E_{\alpha})\oplus(\ker\left.  \alpha\right\vert_{\mathfrak{h}}\cap\mathfrak{k})$, i.e., $\mathbf{d}_{\alpha}\beta$ is real.

(3) Suppose $\mathbf{d}_{\alpha}\beta\pm\mathbf{d}_{\alpha}\gamma$ is a root. Then $\mathbf{d}_{\alpha}^{-1}(\mathbf{d}_{\alpha}\beta\pm\mathbf{d}_{\alpha }\gamma)=\beta\pm\gamma$ is also a root, a contradiction. 

(4) One may compute directly that $Ad(\exp\frac{i\pi}{4}(\theta E_{\alpha }-E_{\alpha}))E_{\beta}=E_{\beta},$ since $\theta E_{\alpha}\in\mathfrak{g}_{-\alpha}$ and $\beta\pm\alpha$ not a root implies $\mathfrak{g}_{\beta\pm\alpha}=\{0\}$, that is, that $[E_{\alpha},E_{\beta}]=[\theta E_{\alpha},E_{\beta}]=0$.
\end{proof}

\bigskip
This section culminates with the next theorem, which reveals the structure of
the CSA $\mathfrak{h}_{\mathbf{F}}$ associated to an admissible root system
$\mathbf{F}$.

\begin{theorem}
\label{lemma:hj-decomp} Let $\mathbf{F}=\{\alpha_{1},\dots,\alpha_{j}\}$ be a real admissible root system associated to the $\theta$-stable CSA $\mathfrak{h}_{\mathbf{F}}.$ Then
\begin{equation}
\mathfrak{h}_{\mathbf{F}}=\bigcap\nolimits_{m=1}^{j}\ker\left.  \alpha_{m}\right\vert _{\mathfrak{h}_{0}} \oplus\bigoplus\nolimits_{m=1}^{j}\mathbb{R}(E_{\alpha_{m}}+\theta E_{\alpha_{m}}). \label{eqn:hj-decomp}
\end{equation}
\end{theorem}

\begin{proof}
Let $\gamma_{n}$ denote the image of $\alpha_{n}$ under the sequence of Cayley transforms associated to the ordered $(n-1)$-tuple $(\alpha_{1},\alpha_{2},\dots,\alpha_{n-1}),$ that is,
\[
\gamma_{j}=\mathbf{d}_{\mathbf{d}_{\ddots_{\mathbf{d}_{\alpha_{1}}\alpha_{2}}{\iddots}}\alpha_{n-1}}\alpha_{n}.
\]
Note that by repeatedly applying Lemma \ref{lemma:Cayley_props}(2), each $\gamma_{n}$ is a real root. Indeed, $\gamma_{n}=\alpha_{n}$ on their common domain of definition, and $\gamma_{n}$ is zero otherwise.

We proceed by induction on the length of $\mathbf{F}$. For $\mathbf{F}=\{\alpha_{1}\}$, the lemma is just equation (\ref{eqn:Cayley_CSA}). Note that by Lemma \ref{lemma:Cayley_props}(2) and (3), after applying the Cayley transforms associated to the roots $\alpha_{1},\dots,\alpha_{n-1},$ the roots $\gamma_{n}$ is real with respect to the new CSA. Suppose then that (\ref{eqn:hj-decomp}) is true for $\mathbf{F}=\{\alpha_{1},\dots,\alpha_{n-1}\}.$ Then by (\ref{eqn:Cayley_CSA}) we have
\begin{equation}
\mathfrak{h}_{\{\alpha_{1},\dots,\alpha_{n}\}}=\ker\left.  \gamma_{n}\right\vert _{\mathfrak{h}_{(\alpha_{1},\dots,\alpha_{n-1}\}}} \oplus\mathbb{R}(E_{\gamma_{n}}+\theta E_{\gamma_{n}}). \label{eqn:hj-decomp1}
\end{equation}
Since $\{\alpha_{1},\dots,\alpha_{j}\}$ is strongly orthogonal, \ref{lemma:Cayley_props} parts (3) and (4) imply
\[
E_{\gamma_{n}}=E_{\mathbf{d}_{\mathbf{d}_{\ddots_{\mathbf{d}_{\alpha_{1}}\alpha_{2}}\iddots}\alpha_{n-1}}\alpha_{n}}=E_{\alpha_{n}}
\]
for all $n$, so that $\mathfrak{h}_{\{\alpha_{1},\dots,\alpha_{n}\}} =\ker\gamma_{n}\oplus\mathbb{R}(E_{\alpha_{n}}+\theta E_{\alpha_{n}}).$

Now, each vector $E_{\alpha_{m}}+\theta E_{\alpha_{m}},~m=1,\dots,n-1,$ is in $\mathfrak{k}$ and $\mathfrak{h}_{\{\alpha_{1},\dots,\alpha_{n-1}\}}$ (by our induction hypothesis). Since $\gamma_{n}$ is real, it vanishes on $\mathfrak{k}\cap\mathfrak{h}_{\{\alpha_{1},\dots,\alpha_{n-1}\}}$, whence
\begin{equation}
\ker\gamma_{n}=\ker\left.  \gamma_{n}\right\vert _{\bigcap\nolimits_{m=1}^{n-1}\ker\gamma_{m}} \oplus\bigoplus\nolimits_{m=1}^{n-1}\mathbb{R} (E_{\alpha_{m}}+\theta E_{\alpha_{m}}). \label{eqn:hj-decomp2}
\end{equation}
Now, repeatedly applying Lemma \ref{lemma:Cayley_props}(1), we see that $\left.  \gamma_{n}\right\vert _{\bigcap\nolimits_{m=1}^{n-1}\ker\alpha_{m}}=\left.  \alpha_{n}\right\vert _{\bigcap\nolimits_{m=1}^{n-1}\ker\alpha_{m}}$
so that
\[
\ker\left.  \gamma_{n}\right\vert _{\bigcap\nolimits_{m=1}^{n-1}\ker\alpha_{m}}=\bigcap\nolimits_{m=1}^{n}\ker\alpha_{m}.
\]
Combining this with (\ref{eqn:hj-decomp1}) and (\ref{eqn:hj-decomp2}) completes the theorem.
\end{proof}

\section{Isotropic foliations\label{sec:foliations}}

This section forms the heart of this article. Here, we describe how the Iwasawa decomposition, in particular the nilpotent part, induces a certain foliation of the each regular (co)adjoint orbit $\Omega=Ad_{G}X$. In some cases, this foliation is actually a Lagrangian fibration admitting a transverse Lagrangian section, so that the orbit is symplectomorphic to a cotangent bundle. In all cases, the foliations are isotropic and left invariant.

Depending on the regular orbit, the induced foliation may be the trivial one in which each leaf is just a point. This is not a failure of our method, but rather a reflection of the fact that some orbits do not admit nontrivial left-invariant isotropic foliations; for example, the orbits in $\mathfrak{sl}(2,\mathbb{R})$ with stabilizer algebra conjugate to $\mathbb{R}\left(
\begin{smallmatrix}
0 & 1\\
-1 & 0
\end{smallmatrix}
\right)  $. On the other hand, we can prove that except for low rank Lie algebras, at least two classes of regular orbits admit nontrivial left-invariant nilpotent isotropic foliations.

In the first part, we consider the ``best'' case: regular orbits whose stabilizer algebra is a maximally noncompact CSA (this includes all orbits in complex algebras considered as real algebras), that is, type-$0$ orbits. In this case, we obtain foliations with leaves of the largest dimension. In fact, the tangent space to the foliation at $X$ is the nilpotent subalgebra. If the Lie algebra is a split real form or a complex simple Lie algebra regarded as a real Lie algebra, then the isotropic foliation is in fact a Lagrangian fibration. We will conclude the first part by briefly discussing Darboux frames.

In the second part, we consider type-$\mathbf{F}$, $|\mathbf{F}|>0$, regular orbits (i.e., those whose stabilizers are not maximally noncompact). We will sees that a subalgebra of the nilpotent algebra will induce a left-invariant foliation. We will also see that higher compact dimension of the stabilizer of the orbit corresponds to lower dimension of the leaves. Moreover, if the rank of $\mathfrak{g}$ is not too small, regular type-$\mathbf{F}$ orbits with $|\mathbf{F}|=1$ admit nontrivial left-invariant isotropic foliations.

\subsection{Maximally noncompact orbits.\label{subsec:maxNC}}

Let $X$ be a regular point of $\mathfrak{g}$, and assume without loss of generality that $G_{X}$ is $\theta$-stable. If $\Omega=Ad_{G}X$ is a type-$0$ orbit, we will call $\Omega$ a \textit{maximally noncompact orbit}.

Note that this condition is automatically satisfied by coadjoint orbits in complex simple Lie algebras or any noncomplex Lie algebras which have a unique conjugacy class of CSAs. In this section, we describe how for such orbits, the Iwasawa decomposition yields directly a left-invariant isotropic foliation of the orbit. For nonmaximally noncompact orbits there may still be left-invariant isotropic foliations arising from the Iwasawa decomposition, but as the construction is more complicated, we defer it to the next section.

\begin{theorem}
\label{thm:foliation}Each maximally noncompact orbit $\Omega=Ad_{G}X$ admits a left-invariant isotropic foliation $\mathcal{N}$ with $\mathcal{N}_{X}=\mathfrak{n}$.
\end{theorem}

\begin{proof}
To begin, if $\mathfrak{g}_{X}$ is a maximally noncompact $\theta$-stable CSA, then it can be decomposed as
\[
\mathfrak{g}_{X}=\mathfrak{t}\oplus\mathfrak{a}
\]
where $\mathfrak{a\subset p}$ is maximal abelian and $\mathfrak{t}\subset Z_{\mathfrak{k}}(\mathfrak{a})$ is maximal abelian. Let $\mathfrak{k}^{\prime}:=\mathfrak{k/t}$. Using the Iwasawa decomposition of $\mathfrak{g}$ relative to $\mathfrak{a}$, we obtain (as vector space)
\[
T_{X}\Omega\simeq\mathfrak{g}/\mathfrak{g}_{X}=(\mathfrak{n}\oplus\mathfrak{a}\oplus\mathfrak{k}) /(\mathfrak{a}\oplus\mathfrak{k})\simeq \mathfrak{n}\oplus\mathfrak{k}^{\prime},
\]
hence we can identify $\mathfrak{n}$ as a subspace of $T_{X}\Omega.$

Next we need to show that $\mathfrak{n}$ is an isotropic subspace of $T_{X}\Omega$. For each $U,V\in\mathfrak{n}$, we have $[U,V]\in\mathfrak{n}=\bigoplus_{\lambda\in\Sigma^{+}}\mathfrak{g}_{\lambda}$. On the other hand, $X\in\mathfrak{g}_{X}\subset\mathfrak{g}_{0}$. Since the restricted root space decomposition is $B$-orthogonal, we obtain
\[
\omega_{X}(U,V)=B(X,[U,V])=0
\]
so that $\mathfrak{n}$ is isotropic.

For each $H\in\mathfrak{g}_{X}\subset\mathfrak{g}_{0},$ the bracket relations $[\mathfrak{g}_{\alpha},\mathfrak{g}_{\beta}]\subseteq\mathfrak{g}_{\alpha+\beta}$ implies $[H,\mathfrak{g}_{\lambda}]\subset\mathfrak{g}_{\lambda}$ so that $[\mathfrak{g}_{X},\mathfrak{n}]\subset\mathfrak{n}$ as desired.

Finally, the distribution $\mathcal{N}$ is involutive since $\mathfrak{n}$ is a subalgebra (\textit{i.e.}, $\mathfrak{n}$ satisfies Frobenius integrability).
\end{proof}

\bigskip
\noindent\textbf{Remark.} Any two choices of $\mathfrak{n}$, which must necessarily arise from two choices of positive roots, are conjugate by the Weyl group $W(\mathfrak{g,h})$, and $W(\mathfrak{g,h})$ acts transitively on the set of such choices. Therefore, there are $\left\vert W(\mathfrak{g},\mathfrak{h})\right\vert $ distinct left-invariant isotropic foliations of a maximally noncompact orbit $\Omega$.

\bigskip
Recall that a real Lie algebra is said to be a \emph{split real form} if it admits a CSA $\mathfrak{h}$ such that the roots $\Delta(\mathfrak{g}^{\mathbb{C}},\mathfrak{h}^{\mathbb{C}})$ take real values on $\mathfrak{h}$. In particular, if $\mathfrak{g}$ is a split real form, such CSAs are the maximal abelian subspaces in~$\mathfrak{p}$. Hence, the maximally noncompact $\theta$-stable CSA $\mathfrak{h}_{0}$ is a maximal abelian in $\mathfrak{p}$ (that is, $\mathfrak{h}_{0}=\mathfrak{a}$). For a split real form $\mathfrak{g}$, the restricted root system $\Sigma(\mathfrak{g},\mathfrak{a} =\mathfrak{h}_{0})$ coincides with the root system $\Delta(\mathfrak{g}^{\mathbb{C}},\mathfrak{h}_{0}^{\mathbb{C}})$ of the complexification. Similarly, if $\mathfrak{g}$ is a complex simple Lie algebra regarded as a Lie algebra, the restricted roots correspond to the roots of $\mathfrak{g}$ as a complex Lie algebra. In both cases, the global Iwasawa decomposition induces more structure on the maximally noncompact regular orbits.

\begin{theorem}
\label{thm:splitrealLag}If $\mathfrak{g}$ is a split real form or a complex simple Lie algebra regarded as a real Lie algebra, then the left-invariant isotropic foliation of Theorem \ref{thm:foliation} is a left $G$-invariant Lagrangian $\mathbb{R}^{n}$-fibration over $Ad_{K}X,$ where $n=\dim \mathfrak{n}=\frac{1}{2}\dim\Omega$ and $Ad_{K}X$ is a transverse Lagrangian submanifold.
\end{theorem}

\begin{proof}
If $\mathfrak{g}$ is a split real form or a complex Lie algebra, then $Z_{\mathfrak{k}}(\mathfrak{a})=\{0\}$ which means that $\mathfrak{g}_{X}$ is itself a maximal abelian subalgebra of $\mathfrak{p}$. Hence, the restricted roots are just the roots of $\mathfrak{g}$ with respect to $\mathfrak{g}_{X}.$

To show that $\mathfrak{n}$ is a Lagrangian subspace, it is enough to show that $\dim\mathfrak{n}=\dim\mathfrak{k}$. This follows, though, from the root space decompositions
\[
\mathfrak{k=}\bigoplus_{\alpha\in\Delta^{+}}\tfrac{1}{2}(1+\theta)\mathfrak{g}_{\alpha}\text{ and }\mathfrak{n=}\bigoplus_{\alpha\in\Delta^{+}}\mathfrak{g}_{\alpha}
\]
since for every $\alpha$, $\dim\mathfrak{g}_{\alpha}=1$ when $\mathfrak{g}$ is a split real form and $\dim\mathfrak{g}_{\alpha}=2$ when $\mathfrak{g}$ is a complex Lie algebra regarded as a real Lie algebra. Hence, the isotropic foliation of $\Omega$ is a Lagrangian foliation.

Since $\mathfrak{n}$ is nilpotent and $N$ is simply connected, an $Ad_{N}$-orbit in $\mathfrak{g}$ is diffeomorphic to $\mathbb{R}^{n},$ where $n=\dim\mathfrak{n}$. Now we can appeal to the global Iwasawa decomposition; we claim that if $Y\in\Omega$, then there exist $k\in K$ and $n\in N$ such that $Y=Ad_{k}Ad_{n}X$. Indeed, if $Y=Ad_{g}X$, decompose $g=kna$ for some $k\in K,$ $n\in N,$ and $a\in A$. Then $Ad(g)X=Ad(k)Ad(n)Ad(a)X=Ad(k)Ad(n)X.$ Hence, the leaves of the Lagrangian foliation are $Ad_{k}Ad_{N}X$, and each leaf passes through a unique point $Ad_{k}X\in Ad_{K}X$.

To show that $Ad_{K}X$ is a Lagrangian submanifold, it suffices to show that $\mathfrak{k}$ is an isotropic subspace of $\mathfrak{k}\oplus\mathfrak{n}$. To this end, let $U,V\in\mathfrak{k}$. Then $[U,V]\in\mathfrak{k}$. Since $X\in\mathfrak{g}_{X}\subset\mathfrak{p}$,
\[
\omega_{X}(U,V)=B(X,[U,V])=0
\]
since the Cartan decomposition is $B$-orthogonal.
\end{proof}

\bigskip
Applying \cite[Thm 7.1]{Weinstein71} to the above theorem elucidates the structure of these maximally noncompact orbits yet further.

\begin{corollary}
If $\mathfrak{g}$ is a split real form or a complex simple Lie algebra regarded as a real algebra, then each maximally noncompact orbit $\Omega=Ad_{G}X$ is symplectomorphic to $T^{\ast}(Ad_{K}X).$
\end{corollary}

\noindent\textbf{Remarks.}
\begin{t_enumerate}
\item The real simple noncompact noncomplex Lie algebras which are split real forms are $\mathfrak{sl}(n,\mathbb{R)}$, $\mathfrak{so}(n,n),~\mathfrak{so}(n,n+1),~\mathfrak{sp}(n,\mathbb{R}),~$(for all $n$) as well as the exceptional algebras $E\text{\textit{I}},~E\text{\textit{V}},~E\text{\textit{VIII}},~F\text{\textit{I}}$ and $G$. See Figure \ref{fig:LAdata} in the Appendix for more details.
\item As mentioned in the remarks following Theorem \ref{thm:foliation}, the choice of $\mathfrak{n}$ in an Iwasawa decomposition is unique up to a choice of positive restricted roots. The set of such choices is parameterized by the Weyl group of the restricted root system. Each such choice yields a Lagrangian fibration of a maximally noncompact orbit, and hence a symplectomorphism of the orbit with $T^{\ast}\left(  Ad_{K}X\right)  $.
\item By left invariance we see that in fact $\Omega\simeq Ad_{N}\times Ad_{K}X.$ Also, because the stabilizer of the coadjoint action on $\Omega$ is $A$, both $K$ and $N$ act freely, and we see in fact that $\Omega\simeq N\times K$.
\end{t_enumerate}

\bigskip
The next result shows that in the case of a split real form, since the restricted roots coincide with the roots and the root spaces are one dimensional, the root space decomposition of $\mathfrak{k}$ and $\mathfrak{n}$ induces Darboux coordinates on $T_{X}\Omega$. Left translation then provides a global Darboux frame.

\begin{theorem}
\label{thm:Dcoords} Suppose $\mathfrak{g}$ is a split real form, and $X$ is a maximally noncompact regular element which generates a $\theta$-stable CSA $\mathfrak{h}$. For each root space $\mathfrak{g}_{\alpha}$, choose a nonzero vector $X_{\alpha}\in\mathfrak{g}_{\alpha}\cap\mathfrak{g}.$ For each positive root $\alpha\in\Delta^{+},$ define 
\[
x_{\alpha}:=\sqrt{\alpha(X)B(X_{\alpha},X_{-\alpha})}^{-1}X_{\alpha}\text{\ and }y_{\alpha}:=\sqrt{\alpha(X)B(X_{\alpha},X_{-\alpha})}^{-1}X_{-\alpha}.
\]
Then $\{x_{\alpha},x_{\alpha}+y_{\alpha}\}_{\alpha\in\Delta^{+}}$ is a Darboux basis for $T_{X}\Omega$ such that $\mathfrak{n}=\operatorname*{span}\{x_{\alpha}\}$ and $\mathfrak{k}=\operatorname*{span}\{x_{\alpha}+y_{\alpha}\}$.
\end{theorem}

\begin{proof}
First, observe that since $\mathfrak{g}$ is a split real form, all of the roots are real whence the vectors $X_{\alpha}$ exist. Since the Killing form is $Ad$-invariant and $X\in\mathfrak{h}$ implies $[X,X_{\alpha}]=\alpha(X)X_{a},$ we have $B(X,[X_{\alpha},X_{-\alpha}])=-B([X_{\alpha},X],X_{-\alpha})=\alpha(X)B(X_{\alpha},X_{-\alpha})$ and 
\[ 
\omega_{X}(x_{\alpha},x_{\alpha}+y_{\alpha}) =\omega_{X}(x_{\alpha},y_{\alpha}) =\frac{B(X,[X_{\alpha},X_{-\alpha}])}{\alpha(X)B(X_{\alpha},X_{-\alpha})}=1.
\]
Moreover, for $\alpha\neq\beta$ we get $\omega_{X}(x_{\alpha},x_{\beta}+y_{\beta})=0$ and $\omega_{X}(x_{\alpha},x_{\beta})=\omega_{X}(x_{\alpha}+y_{\alpha},x_{\beta}+y_{\beta})=0$ since $\mathfrak{h}$ is $B$-orthogonal to the root spaces $\mathfrak{g}_{\pm\alpha\pm\beta}$.
\end{proof}

\bigskip
\noindent\textbf{Remark.} We can also form the triangular decomposition $\mathfrak{g=n}\oplus\mathfrak{a}\oplus\mathfrak{n}_{-}$, where 
\[
\mathfrak{n}_{-}:=\bigoplus_{\alpha\in\Delta^{+}}\mathfrak{g}_{-\alpha}.
\]
For a split real form $\mathfrak{g}$ and regular element $X$ which generates a maximally noncompact CSA, the subalgebras $\mathfrak{n}$ and $\mathfrak{n}_{-}$ both induce Lagrangian foliations of $\Omega$ (although there is no leaf of either that is a global transverse section for the other), and the basis $\{x_{\alpha},y_{\alpha}\}$, with $x_{\alpha}$ and $y_{\alpha}$ as in the preceding theorem, forms a Darboux basis for $T_{X}\Omega$ such that $\mathfrak{n}=\operatorname*{span}\{x_{\alpha}\}$ and $\mathfrak{n}_{-}=\operatorname*{span}\{y_{\alpha}\}$. 

\subsection{Nonmaximally noncompact orbits.\label{subsec:nonmaxNC}}

In this section, we describe how a succession of Cayley transforms applied to maximally noncompact CSAs induces left-invariant isotropic foliations of type-$\mathbf{F},~|\mathbf{F}|>0,$ coadjoint orbits. Since there is only one orbit type for complex simple Lie algebras (and hence is trivially maximally compact), we will focus our attention in this section on a noncompact noncomplex real simple Lie algebra $\mathfrak{g}$. Fix a maximally noncompact $\theta$-stable CSA $\mathfrak{h}_{0}.$

Let $\mathfrak{h}_{\mathbf{F}}$ be a Cartan subalgebra associated to the real admissible root system $\mathbf{F}=\{\alpha_{1},\dots,\alpha_{j}\}$. According to Section \ref{sec:cayley}, and in particular Lemma \ref{lemma:Cayley_props}, applying successive Cayley $D$-transforms associated to the sequence $\alpha_{1},\dots,\alpha_{j}$ yields a sequence $\mathfrak{h}_{1},\dots,\mathfrak{h}_{j}$ of nonconjugate $\theta$-stable CSAs such that $\mathfrak{p}$-$\dim\mathfrak{h}_{n}=\mathfrak{p}$-$\dim\mathfrak{h}_{n-1}-1$ and $\mathfrak{h}_{j}$=$\mathfrak{h}_{\mathbf{F}}$.

\bigskip
\noindent\textbf{Remark.} Even when considering noncompact noncomplex real simple Lie algebras, it happens in some cases that there is a unique conjugacy class of CSAs, e.g. $\mathfrak{sl}(n,\mathbb{H})$. In these cases, the CSAs are necessarily type-$0$, and so are covered by the construction in the previous section.

\bigskip
The next theorem describes how left-invariant isotropic foliations of orbits of type-$\mathbf{F}$ for $\mathbf{F}=\{\alpha_{1},\dots,\alpha_{j}\}$ are related to those of type $\{\alpha_{1},\dots,\alpha_{j-1}\}$. For each $n=1,\dots,j$, let $\beta_{n}=\left.  \alpha_{n}\right\vert _{\mathfrak{a}}$ be the restriction of $\alpha_{n}$ to $\mathfrak{a}$ (so $\beta_{n}\in \Sigma^{+}$ is a positive restricted root).

\begin{theorem}
\label{thm:general case}
Let $\Sigma_{0}=\Sigma^{+}.$ For each $n=1,\dots,j$, let $\Sigma_{n}$ be the largest subset of $\Sigma_{n-1}$ such that
\begin{t_enumerate}
\item for each $m=1,\dots,n,$ if $\mu\in\Sigma_{n}$ and $\mu\pm\beta_{m}\subset\Sigma$, then $\mu\pm\beta_{m}\in\Sigma_{n}$,
\item for each $m=1,\dots,n,$ we have $\pm\beta_{m}\not \in (\Sigma_{n}+\Sigma_{n})\cap\Sigma$, and
\item $(\Sigma_{n}+\Sigma_{n})\cap\Sigma\subset\Sigma_{n}$.
\end{t_enumerate}
Then $\mathfrak{n}_{n}:=\bigoplus\nolimits_{\mu\in\Sigma_{n}}\mathfrak{g}_{\mu}$ is an $\mathfrak{h}_{\{\alpha_{1},\dots,\alpha_{n}\}}$-stable subalgebra of $\mathfrak{g}$. Moreover, if $X\in\mathfrak{g}$ is a regular element such that $G_{X}=\mathfrak{h}_{\{\alpha_{1},\dots,\alpha_{n}\}}$, then $\mathfrak{n}_{n}$ is $\omega_{X}$-isotropic subspace of $T_{X}(Ad(G)X)$ and hence induces a left-invariant isotropic foliation $\mathcal{N}_{n}$ of $Ad(G)X$ with $\left(  \mathcal{N}_{n}\right)  _{X}=\mathfrak{n}_{n}.$
\end{theorem}

\begin{proof}
To show that $\mathfrak{n}_{n}$ is $\mathfrak{h}_{n}$-stable, first consider $X\in\bigcap\nolimits_{m=1}^{n}\ker\alpha_{m}\subset\mathfrak{h}_{0}$. Since $\mathfrak{h}_{0}\subset\mathfrak{g}_{0}$, we have $[X,\mathfrak{g}_{\mu }]\subset\mathfrak{g}_{\mu}$ for all restricted roots $\mu$, whence $[X,\mathfrak{n}_{n}]\subset\mathfrak{n}_{n}.$ Next, consider $X=E_{\alpha_{m}}+\theta E_{\alpha_{m}}$. Since $\alpha_{m}$ is real, $\theta E_{\alpha_{m}}\in\mathfrak{g}_{-\alpha_{m}}$, so $[X,\mathfrak{g}_{\mu}]\in\mathfrak{g}_{\mu+\alpha_{m}}\oplus\mathfrak{g}_{\mu-\alpha_{m}}$. By condition (1), we see that if $\mu\in\Sigma_{n}$ then $\mathfrak{g}_{\mu \pm\beta_{m}}\subset\mathfrak{n}_{n}$ which implies $[X,\mathfrak{g}_{\mu }]\subset\mathfrak{n}_{n}$. Since this holds for each $m=1,\dots,n$, the direct sum decomposition of Theorem \ref{lemma:hj-decomp} implies $[\mathfrak{h}_{n},\mathfrak{n}_{n}]\subset\mathfrak{n}_{n}$ as desired.

Next, we observe that since $[\mathfrak{g}_{\mu},\mathfrak{g}_{\nu}]\subset\mathfrak{g}_{\mu+\nu}$, condition (3) implies that $\mathfrak{n}_{n}$ is a subalgebra.

Finally, to see that $\mathfrak{n}_{n}$ is an $\omega_{X}$-isotropic subspace of $T_{X}(Ad_{G}X)\simeq\mathfrak{n}\oplus(\mathfrak{a}\oplus\mathfrak{k})/\mathfrak{h}_{n}$, let $U,V\in\mathfrak{n}_{n}.$ Let $X=X_{\ker}+\sum_{m=1}^{n}c_{m}(E_{\alpha_{m}}+\theta E_{\alpha_{m}})$ be the decomposition of $X$ according to Theorem \ref{lemma:hj-decomp}. Recall that $\mathfrak{g}_{\mu}$ is orthogonal to $\mathfrak{g}_{\nu}$ unless $\nu=\pm\mu$. Thus, since $X_{\ker}\in\mathfrak{g}_{0}$, we have $B(X,[\mathfrak{n}_{n},\mathfrak{n}_{n}])=0.$ Moreover, condition (2) insures that $\mathfrak{g}_{\pm\beta_{m}}\perp\mathfrak{n}_{n}$ which implies $E_{\alpha_{m}}+\theta E_{\alpha_{m}}\perp\lbrack\mathfrak{n}_{n},\mathfrak{n}_{n}]$ for each $m=1,\dots,n$ whence we conclude that $\omega_{X}(U,V)=B(X,[U,V])=0.$
\end{proof}

\bigskip
\noindent\textbf{Remarks.}
\begin{t_enumerate}
\item It may happen that $\mathfrak{n}_{j}=\{0\}$, in which case the isotropic foliation is of $\Omega$ is just the trivial one in which each leaf is a single point. 
\item The construction of $\mathfrak{h}_{\mathbf{F}}$ from $\mathfrak{h}_{0}$ does not depend on the order in which the Cayley transforms associated to $\{\alpha_{1},\dots,\alpha_{j}\}$ are applied to $\mathfrak{h}_{0}$. On the other hand, it might be possible that one can choose a an order which maximizes the dimension of the the resulting isotropic nilpotent subalgebra foliation (we do not know). But even then it is sometimes unavoidable that $\mathfrak{n}_{j}=0$. This happens, for example, for coadjoint orbits in $\mathfrak{sl}(2,\mathbb{R)}$ with maximally compact stabilizers. On the other hand, as we will see in the next section, we \emph{can} sometimes guarantee that the left-invariant isotropic foliation of Theorem \ref{thm:general case} is nontrivial.
\end{t_enumerate}

\subsubsection{Type-$\mathbf{\{\alpha\}}$ orbits.\label{subsec:type-1}}

If $\mathfrak{g}$ admits more than one conjugacy class of CSAs, and if $\mathfrak{g}$ is not too small, then we can show that when $\alpha$ is a simple positive root, type-$\{\alpha\}$ regular orbits admit nontrivial left-invariant isotropic foliations. The reason is that if the number of positive restricted roots is greater than $1$, when the restricted root system is reduced,\footnote{Recall that a root system is \emph{reduced} if for each restricted root $\mu$, $\frac{1}{2}\mu$ is not a root.} or greater than $2$, when the restricted root system is not reduced, it will follow that $\Sigma_{1}$ is nonempty ($\Sigma_{1}$ is the subset of positive restricted roots appearing in Theorem \ref{thm:general case} which controls the left-invariant isotropic foliation of type-$\{\alpha\}$ orbits). Hence, $\Sigma_{1}$ induces a nontrivial left-invariant isotropic foliation. 

We first show that when there is more than one conjugacy class of Cartan subalgebras, positive real roots always exist, and hence there are always nonmaximally compact orbits of type $\{\alpha\}$, for $\alpha>0$ simple.

\begin{lemma}
\label{lemma:alpha1-simple} Suppose there exists a real root $\beta$. Then there exists at least one positive simple real root $\alpha$.
\end{lemma}

\begin{proof}
If $\beta$ is real, then so is $-\beta$, so we may assume without loss of generality that $\beta$ is positive. Decompose $\beta$ into simple positive simple roots
\begin{equation}
\beta=\sum_{\alpha\in\Pi}c_{\alpha}\alpha. \label{eqn:beta_decomp}
\end{equation}
Since $\beta$ is positive, all of the coefficients $c_{\alpha}$ are positive. Since $\beta$ is real, $\beta(\mathfrak{t})=0$ which implies $\alpha(\mathfrak{t})=0$ for each $\alpha$ with $c_{\alpha}\neq0$, that is, each $\alpha$ appearing in the decomposition (\ref{eqn:beta_decomp}) is also real.
\end{proof}

\bigskip
\begin{theorem}
\label{thm:h1iso}
Assume that $\alpha$ is simple, and let $\beta$ denote the restriction to $\mathfrak{a}$ of $\alpha.$ If the restricted root system $\Sigma$ is reduced, let $\Sigma_{1}=\Sigma^{+}\setminus\{\beta\}$. If $\beta$ and $2\beta$ are restricted roots, let $\Sigma_{1}=\Sigma^{+}\setminus \{\beta,2\beta\}$. Define $\mathfrak{n}_{1}:=\bigoplus\nolimits_{\mu\in \Sigma_{1}}\mathfrak{g}_{\mu}\subset\mathfrak{n}$. If $\Omega=Ad(G)X$ is type-$\{\alpha\}$, then it admits a left-invariant isotropic foliation $\mathcal{N}_{1}$ with $\left(  \mathcal{N}_{1}\right)  _{X}=\mathfrak{n}_{1}.$
\end{theorem}

\begin{proof}
First, we remark that it is easy to see that since $\beta$ is simple, $\mathfrak{n}_{1}$ is a subalgebra. Thus, we need only show that $\Sigma_{1}$ satisfies properties (1)--(3) of Theorem \ref{thm:general case}.

For the first property, let $\mu\in\Sigma_{1}.$ Since $\beta$ is a positive simple root, if $\mu+\beta$ is a root then it is a positive root, whence $\mu+\beta\in\Sigma_{1}$ since $\mu+\beta\neq\beta$ and $\mu+\beta=2\beta$ only if $\mu=\beta\not \in \Sigma_{1}$. Next, decompose $\mu$ into simple positive roots $\mu=c_{\beta}\beta+\sum_{\nu\in\Sigma^{+}\backslash\{\beta\}}c_{\nu}\nu,$ where $c_{\nu}\geq0$ for all $\nu\in\Sigma^{+}.$ If $\mu-\beta$ is a root, then it can be decomposed as a sum over positive simple roots, and the coefficients all have the same sign, whence it must be positive since $\mu$ is positive. In particular, the coefficient $c_{\beta}^{\prime}$ of $\beta$ is $\geq0.$ This means that $\mu=(c_{\beta}-1)\beta+\sum_{\nu \in\Sigma^{+}\backslash\{\beta\}}c_{\nu}\nu$ for some $c_{\beta}\geq1$, whence $\mu-\beta$ is positive and $\neq0,\beta$ (since $\mu\neq\beta,2\beta$), that is, $\mu-\beta\in\Sigma_{1}$.

Property (2) is a consequence of the fact that $\beta$ is a positive simple root, and thus cannot be expressed as a positive linear combination of positive roots.

Property (3) is also a consequence of the simplicity of $\beta.$ If $\mu, \nu\in\Sigma_{1}$ such that $\mu+\nu$ is a restricted root, then $\mu+\nu$ is a positive restricted root (since $\mu,\nu$ are both positive), though not a simple one. Hence $\mu+\nu\neq\beta.$ Of course, $\mu+\nu$ cannot equal $2\beta$ either, since otherwise $\mu=\nu=\beta\not \in \Sigma_{1}.$
\end{proof}

\bigskip
\noindent\textbf{Remarks.}
\begin{t_enumerate}
\item If $|\Sigma^{+}|>1$ and $2\beta$ is not a restricted root, then $\Sigma^{+}\setminus\{\beta\}$ is nonempty and $\dim\mathfrak{n}_{2}>0$. If $2\beta$ is a restricted root, then we must have $|\Sigma^{+}|>2$ in order to insure that $\dim\mathfrak{n}_{1}>0$. Either way, it only happens that $\dim\mathfrak{n}_{1}=0$ for small rank Lie algebras, for example $\mathfrak{sl}(2,\mathbb{R})$.
\item On the other hand, $\dim\mathfrak{n}_{1}<\dim\mathfrak{n}$ since we must, at the very least, ``remove'' $\alpha$ from $\mathfrak{n}$ to insure that the resulting subalgebra is $\mathfrak{h}_{\{\alpha\}}$-stable. To see why, note that $E_{\alpha}+\theta E_{\alpha}\in\mathfrak{h}_{\{\alpha\}}$ implies $[E_{\alpha}+\theta E_{\alpha},E_{\alpha}]=[\theta E_{\alpha},E_{\alpha}]\propto H_{\alpha}\in\lbrack\mathfrak{h}_{\{\alpha\}},\mathfrak{n]}$. But $H_{\alpha}\not \in \mathfrak{h}_{\{\alpha\}} \oplus\mathfrak{n}$ by construction.
\end{t_enumerate}

\section{Rulings\label{sec:rulings}}

Our last result is that the leaves of the left-invariant isotropic foliations of the previous section are affine subspaces of $\mathfrak{g}$. A surface in $\mathbb{R}^{3}$ which is foliated by lines (that is, by actual affine subspaces of $\mathbb{R}^{3},$ not just by $1$-manifolds) is called a \emph{ruled} surface, and the foliation itself is called a \emph{ruling} (a complex surface foliated by complex lines is also said to be \emph{ruled}); for example, the left-invariant isotropic foliations of the maximally noncompact regular orbits in $\mathfrak{sl}(2,\mathbb{R})$ are the classical rulings of the $1$-sheeted hyperboloid. We will say that a foliation of a manifold $M\subset\mathbb{R}^{N}$ is an $n$\emph{-ruling} if the leaves are $n$-dimensional affine subspaces.

\begin{theorem}
Let $\Omega$ be a regular orbit in $\mathfrak{g}$, say of type $\mathbf{F}=\{\alpha_{1},\dots,\alpha_{j}\}.$ The left-invariant isotropic foliation of $\Omega$ of Theorem \ref{thm:general case} (or \ref{thm:foliation}, if $\Omega$ is maximally noncompact) is a $\dim\mathfrak{n}_{j}$-ruling, where $\mathfrak{n}_{j}:=\bigoplus\nolimits_{\mu\in\Sigma_{j}}\mathfrak{g}_{\mu}$ is the nilpotent subalgebra associated to $\mathbf{F}$ constructed in \ref{thm:general case}.
\end{theorem}

\begin{proof}
Let $X\in\Omega$ such that $G_{X}=\mathfrak{h}_{j}$ and $\mathfrak{n}_{j}:=\left(  \mathcal{N}_{j}\right)  _{X}=\oplus_{\mu\in\Sigma_{j}}\mathfrak{g}_{\mu}$ is the subalgebra given by Theorem \ref{thm:general case} of the nilpotent subalgebra $\mathfrak{n}$ of the Iwasawa decomposition. We will first show that the leaf through $X$ is $X+\mathfrak{n}_{j}.$ The result will then follow by left translation of this leaf since $Ad_{g}$ is linear.

To see that the leaf through $X$ is $X+\mathfrak{n}_{j},$ note that since $\mathfrak{n}_{j}$ is a nilpotent subalgebra of $\mathfrak{n}$, there exists an analytic subgroup $N_{j}$ of $G$ which is nilpotent and simply connected, whence the exponential map is a diffeomorphism $\exp:\mathfrak{n}_{j}\rightarrow N_{j}$. We will show that $Ad(N_{j})X=X+\mathfrak{n}_{j}.$ Let $y\in N_{j}.$ Then $y=\exp Y$ for some $Y\in\mathfrak{n}_{j}$. Since $\mathfrak{n}_{j}$ is $\mathfrak{h}_{j}$-stable, we must have $ad(Y)X=[Y,X]\in \mathfrak{n}_{j}.$ Since $\mathfrak{n}_{j}$ is a subalgebra, it follows that $ad_{Y}^{n}X\in\mathfrak{n}_{j}$ for all $n\geq1$. Hence we have
\[
Ad_{y}X=e^{ad_{Y}}X=X+\sum_{n=1}^{\infty}\frac{1}{n!}ad_{Y}^{n}X\in X+\mathfrak{n}_{j}.
\]
(Since $\mathfrak{n}_{j}$ is nilpotent, the sum above is actually finite.) This shows that $Ad_{N_{j}}X\subset X+\mathfrak{n}_{j}.$ On the other hand, since $\mathfrak{n}_{j}\cap\mathfrak{h}_{j}=\{0\}$, we must have $\dim Ad_{N_{j}}X=\dim\mathfrak{n}_{\mathfrak{j}}$ from which it follows that $Ad_{N_{j}}X=X+\mathfrak{n}_{j}.$
\end{proof}

\section*{Acknowledgements}

The author would like to thank A. Caine for several useful discussion during the early phase of this work, and in particular for pointing out that the foliations of type-$0$ orbits in $SL(2,R)$ are the classical rulings of the hyperboloid.

\section*{Appendix: Nilpotent foliation data for simple Lie algebras}

Figure \ref{fig:cplxLAdata} shows the data for simple complex Lie algebras which is relevant to the foliations described in Theorems \ref{thm:foliation} and \ref{thm:general case}, and Figure \ref{fig:LAdata} shows the data for simple real Lie algebras. Note that the simple complex Lie algebras are regarded as real Lie algebras; in particular, the dimensions are \emph{real} dimensions.

The first column of Figure \ref{fig:cplxLAdata} gives the name of the Lie algebra $\mathfrak{g}$ using the conventions of \cite[Sec. I.8]{Knapp}. The second and third columns give the real dimension and the rank (resp.) of the algebra (recall that the rank is the dimension of any CSA). Since the stabilizer algebra of a regular orbit is a CSA, the difference of the second and third columns gives the dimension of the regular orbits in $\mathfrak{g}$, which is recorded in the fourth column. For complex simple Lie algebras, the isotropic foliation induced by the Iwasawa decomposition is a Lagrangian fibration, and so the leaves have half the dimension of the orbit.

\begin{figure}[th]
{\small
\[%
\begin{array}
[c]{|l|c|c|c|}\hline
\makebox[.8in]{$\mathfrak{g}$} & \dim_{\mathbb{R}}\mathfrak{g} & \operatorname*{rank}\mathfrak{g} & \dim\Omega\\\hline\hline
A_{n}:\mathfrak{sl}(n+1,\mathbb{C)} & 2n(n+2) & 2n & 2n(n+1)\\\hline
B_{n}:\mathfrak{so}(2n+1,\mathbb{C}) & 2n(2n+1) & 2n & 4n^{2}\\\hline
C_{n}:\mathfrak{sp}(n,\mathbb{C}) & 2n(2n+1) & 2n & 4n^{2}\\\hline
D_{n}:\mathfrak{so}(2n+1,\mathbb{C}) & 2n(2n-1) & 2n & 4n(n-1)\\\hline
E_{6} & 2\times78=156 & 12 & 144\\\hline
E_{7} & 2\times133=266 & 14 & 252\\\hline
E_{8} & 2\times248=496 & 16 & 480\\\hline
F_{4} & 2\times52=104 & 8 & 96\\\hline
G_{2} & 2\times14=28 & 4 & 24\\\hline
\end{array}
\]
}\caption{Orbit and foliation data for simple complex Lie algebras.}
\label{fig:cplxLAdata}
\end{figure}

Figure \ref{fig:LAdata} below gives data for (noncomplex) simple real Lie algebras. The real rank $\mathbb{R}$-$\operatorname*{rank}\mathfrak{g}$, recorded in the fifth column, is the dimension of a maximal abelian subspace in the noncompact component $\mathfrak{p}$ of a Cartan decomposition $\mathfrak{g}=\mathfrak{k}\oplus\mathfrak{p}$. If the rank of $\mathfrak{g}$ is equal to the real rank of $\mathfrak{g}$, then a maximally noncompact $\theta$-stable CSA is a maximal abelian subspace of $\mathfrak{p}$, whence $\mathfrak{g}$ is a split-real form, and Theorem \ref{thm:splitrealLag} is valid. The sixth column gives the the cardinality of a maximal set of strongly orthogonal real roots, which is equal to the difference in noncompact dimensions of a maximally noncompact algebra and a maximally compact CSA. This number is a lower bound for the number of distinct regular orbit types. The last two columns of Figure \ref{fig:LAdata} record whether $\mathfrak{g}$ is a split real form, and the dimension of the leaves of the isotropic foliation of type-$0$ orbits induced by the Iwasawa decomposition, which is computed as $\dim\mathfrak{n}=\dim\mathfrak{g}-\dim\mathfrak{k}-\mathbb{R}-\operatorname*{rank}\mathfrak{g.}$

\begin{landscape}
\begin{figure}[ht]
{\small
\[%
\begin{array}
[c]{|l|c|c|c|c|c|c|c|c|c|}\hline
\makebox[.8in]{$\mathfrak{g}$} & \Delta & \Sigma & \dim_{\mathbb{R}}\mathfrak{g} & \operatorname*{rank}\mathfrak{g} & \dim\Omega & \mathbb{R}\text{-}\operatorname*{rank}\mathfrak{g} & |\mathbf{F}|_\text{max} & \text{split?} & \dim\mathfrak{n}\\\hline\hline
\mathfrak{sl}(n,\mathbb{R)} & A_{n-1} & A_{n-1} & (n-1)(n+1) & n-1 & n(n-1) & n-1 & [n/2] & \text{yes} & n(n-1)/2\\\hline
\mathfrak{sl}(n,\mathbb{H}) & A_{2n-1} & A_{n-1} & (2n-1)(2n+1) & 2n-1 & 2n(2n-1) & n-1 & 0 &  & 2n(n-1)\\\hline
\mathfrak{su}(p,q),~1\leq p\leq q & A_{p+q-1} & (BC)_{p}\text{ if }p<q & (p+q-1)(p+q) & p+q+1 & (p+q-2)(p+q) & p & p &  & p(2q-1)-2\\
&  & C_{p}\text{ if }p=q &  &  &  &  &  &  & \\\hline
\mathfrak{so}(2p,2q+1),~1\leq p\leq q & B_{p+q} & B_{2p} & (2(p+q)+1)(p+q) & p+q & 2(p+q)^{2} & 2p & 2p & p=q & 4pq\\\hline
\mathfrak{so}(2p,2q+1),~p>q\geq0 & B_{p+q} & B_{2q+1} & (2(p+q)+1)(p+q) & p+q & 2(p+q)^{2} & 2q+1 & 2q+1 &  & 4pq+2(p-q)-1\\\hline
\mathfrak{sp}(p,q),~1\leq p\leq q & C_{p+q} & (BC)_{p}\text{ if }p<q & (2(p+q)+1)(p+q) & p+q & 2(p+q)^{2} & p & p &  & 4pq-p\\
&  & C_{p}\text{ if }p=q &  &  &  &  &  &  & \\\hline
\mathfrak{sp}(n,\mathbb{R}) & C_{n} & C_{n} & n(2n+1) & n & 2n^{2} & n & n & \text{yes} & n^{2}\\\hline
\mathfrak{so}(2p+1,2q+1) & D_{p+q+1} & B_{p}\text{ if }p<q & (2(p+q+1)-1) & p+q+1 & (2(p+q+1)-2) & 2p+1 & 2p &  & 4pq+2q\\
\quad0\leq p\leq q &  & D_{p}\text{ if }p=q & \qquad\times(p+q+1) &  & \qquad\times(p+q+1) &  &  & p=q & \\
\quad\text{not }\mathfrak{so}(1,1)\text{ or }\mathfrak{so}(1,3) &  &  &  &  & &  &  &  & \\\hline
\mathfrak{so}(2p,2q) & D_{p+q} & B_{p}\text{ if }p<q & (2(p+q)-1) & p+q & (2(p+q)-2) & 2p & 2p & p=q & 4pq-2p\\
1\leq p\leq q,\ \text{not }\mathfrak{so}(2,2) &  & D_{p}\text{ if }p=q & \qquad\times(p+q) &  & \qquad\times(p+q) &  &  &  & \\\hline
\mathfrak{so}^{*}(2n) & D_{n} & (BC)_{n(n-1)/2},~n\text{ odd} & n(2n-1) & n & 2n(n-1) & [n/2] & [n/2] &  & n^{2}-n-[n/2]\\
&  & C_{n/2},~n\text{ even} &  &  &  &  &  &  & \\\hline
E\text{\textit{I}} & E_{6} & E_{6} & 78 & 6 & 72 & 6 & 4 & \text{yes} & 36\\\hline
E\text{\textit{II}} & E_{6} & F_{4} & 78 & 6 & 72 & 4 & 4 &  & 34\\\hline
E\text{\textit{III}} & E_{6} & (BC)_{2} & 78 & 6 & 72 & 3 & 2 &  & 30\\\hline
E\text{\textit{IV}} & E_{6} & A_{2} & 78 & 6 & 72 & 3 & 0 &  & 24\\\hline
E\text{\textit{V}} & E_{7} & E_{7} & 133 & 7 & 126 & 7 & 7 & \text{yes} & 63\\\hline
E\text{\textit{VI}} & E_{7} & F_{4} & 133 & 7 & 126 & 4 & 4 &  & 60\\\hline
E\text{\textit{VII}} & E_{7} & C_{3} & 133 & 7 & 126 & 3 & 3 &  & 51\\\hline
E\text{\textit{VIII}} & E_{8} & E_{8} & 248 & 8 & 240 & 8 & 8 & \text{yes} & 120\\\hline
E\text{\textit{IX}} & E_{8} & F_{4} & 248 & 8 & 240 & 4 & 4 &  & 110\\\hline
F\text{\textit{I}} & F_{4} & F_{4} & 52 & 4 & 48 & 4 & 4 & \text{yes} & 24\\\hline
F\text{\textit{II}} & F_{4} & (BC)_{1} & 52 & 4 & 48 & 1 & 1 &  & 15\\\hline
G & G_{2} & G_{2} & 14 & 2 & 12 & 2 & 2 & \text{yes} & 6\\\hline
\end{array}
\]
}
\caption{Orbit and foliation data for simple real Lie algebras.}\label{fig:LAdata}
\end{figure}
\end{landscape}


\bigskip

\begin{thebibliography}{99}                                                                                               %

\bibitem {Abbondandolo-Schwarz}A. Abbondandolo and M. Schwarz.
\newblock On the Floer homology of cotangent bundles.
\newblock {\em Comm. Pure Appl. Math.} \textbf{59} (2006), no. 2, 254--316.

\bibitem {Agaoka-Kaneda}Y.~Agaoka and E.~Kaneda. \newblock Strongly
orthogonal subsets in root systems. \newblock {\em Hokkaido Math. J.} \textbf{31} (2002), no. 1,
107--136.

\bibitem {Arnold-MathMethods} V.~I. Arnol`d. \newblock {\em
Mathematical methods of classical mechanics}, volume~60 of \emph{\ Graduate
Texts in Mathematics}. \newblock Springer-Verlag, New York, 1989. \newblock
Translated from the 1974 Russian original by K. Vogtmann and A. Weinstein,
Corrected reprint of the second (1989) edition.

\bibitem {Bernatska-Holod}J.~Bernatska and P.~Holod.
\newblock Geometry and topology of coadjoint orbits of semisimple {L}ie
groups. \newblock In \emph{Geometry, integrability and quantization},
146--166. Softex, Sofia, 2008.

\bibitem {Duistermaat-Kolk} J.~Duistermaat and J.~Kolk. \newblock
\emph{Lie Groups}. \newblock Springer--Verlag, Berlin, 2000.

\bibitem {Helgason} S.~Helgason. \newblock {\em Differential geometry
and symmetric spaces}. \newblock Pure and Applied Mathematics, Vol. XII.
Academic Press, New York, 1962.

\bibitem {Iwasawa49} K.~Iwasawa. \newblock On some types of topological
groups. \newblock {\em Ann. of Math. (2)} \textbf{50} (1949), 507--558.

\bibitem {Kirillov-Orbit} A.~A. Kirillov. \newblock {\em Lectures on
the orbit method}, volume~64 of \emph{Graduate Studies in Mathematics}.
\newblock American Mathematical Society, Providence, RI, 2004.

\bibitem {Knapp}A.~Knapp. \newblock {\em Lie Groups: Beyond an
Introduction, 2nd Edition}, volume 140 of \emph{Progress in Mathematics}.
\newblock Birkh\"auser, 2002.

\bibitem {KostantCSA}B.~Kostant. \newblock On the conjugacy of real
{C}artan subalgebras. {I}. \newblock {\em Proc. Nat. Acad. Sci. U. S. A.} \textbf{41} (1955), 967--970.

\bibitem {Rothschild} L.~P. Rothschild. \newblock Orbits in a real
reductive {L}ie algebra. \newblock {\em Trans. Amer. Math. Soc.} \textbf{168} (1972), 403--421.

\bibitem {Salamon-Weber} D. Salamon and J. Weber. \newblock Floer
homology and the heat flow. \newblock {\em Geom. Funct. Anal.} \textbf{16}
(2006), no. 5, 1050--1138.

\bibitem {Sugiura} M.~Sugiura. \newblock Conjugate classes of {C}artan
subalgebras in real semi-simple {L}ie algebras. \newblock {\em J.
Math. Soc. Japan}, \textbf{11} (1959), 374--434.

\bibitem {Sugiura-correction} M.~Sugiura. \newblock Correction to my
paper: ``{C}onjugate classes of {C}artan subalgebras in real semisimple {L}ie
algebras''. \newblock {\em J. Math. Soc. Japan}, \textbf{23} (1971), 379--383.

\bibitem {Varadarajan} V.~S. Varadarajan. \newblock {\em Lie groups,
{L}ie algebras, and their representations}, volume 102 of \emph{Graduate Texts
in Mathematics}. \newblock Springer-Verlag, New York, 1984. \newblock Reprint
of the 1974 edition.

\bibitem {Klimyk-Vilenkin}  N.~J. Vilenkin and A.~U. Klimyk. \newblock
\emph{Representation of {L}ie groups and special functions. {V}ol. 3},
volume~75 of \emph{Mathematics and its Applications (Soviet Series)}.
\newblock Kluwer Academic Publishers Group, Dordrecht, 1992.
\newblock Classical and quantum groups and special functions, Translated from
the Russian by V. A. Groza and A. A. Groza.

\bibitem {Viterbo}  C. Viterbo. \newblock
Generating functions, symplectic geometry, and applications.
\newblock {Proceedings of the International
Congress of Mathematicians} Vol. 1, 2 (Z\"{u}rich, 1994), 537--547.
Birkh\"{a}user, Basel, 1995.

\bibitem {Weinstein71}  A.~Weinstein. \newblock Symplectic manifolds and
their {L}agrangian submanifolds. \newblock {\em Advances in Math.} \textbf{6} (1971), 329--346.

\end{thebibliography}
\end{document}